\documentclass[12pt, leqno]{article}
\usepackage{amsmath}
\usepackage{amsthm}
\usepackage{amssymb, latexsym}
\usepackage{enumerate}
\usepackage{url}
\usepackage[arrow, matrix]{xy}

\newtheorem{theorem}{Theorem}
\newtheorem{lemma}[theorem]{Lemma}
\newtheorem{proposition}[theorem]{Proposition}
\newtheorem{corollary}[theorem]{Corollary}
\newtheorem{problem}[theorem]{Problem}

\newcommand{\Z}{\mathbb{Z}}
\newcommand{\Q}{\mathbb{Q}}

\renewcommand{\P}{\mathfrak{P}}

\newcommand{\D}{\mathfrak{D}}
\renewcommand{\O}{\mathcal{O}}

\newcommand{\ol}{\overline}

\begin{document}
\title{Sums of units in function fields II - The extension problem}
\author{Christopher Frei}
\date{}

\maketitle

\begin{abstract}
In 2007, Jarden and Narkiewicz raised the following question: Is it true that each algebraic number field has a finite extension $L$ such that the ring of integers of $L$ is generated by its units (as a ring)? In this article, we answer the analogous question in the function field case. 

More precisely, it is shown that for every finite non-empty set $S$ of places of an algebraic function field $F | K$ over a perfect field $K$, there exists a finite extension $F' | F$, such that the integral closure of the ring of $S$-integers of $F$ in $F'$ is generated by its units (as a ring).
\end{abstract}

\renewcommand{\thefootnote}{}
\footnote{2010 \emph{Mathematics Subject Classification}: Primary 11R58; Secondary 11R27.}
\footnote{\emph{Key words and phrases}: function field, sums of units, generated by units}
\renewcommand{\thefootnote}{\arabic{footnote}}
\setcounter{footnote}{0}

\section{Introduction}\label{introduction}

In their paper \cite{Jarden2007}, Jarden and Narkiewicz proved that, for every finitely generated integral domain $R$ of characteristic $0$ and every positive integer $N$, there exists an element of $R$ that can not be written as a sum of at most $N$ units. This also follows from a result obtained by Hajdu \cite{Hajdu2007}, and applies in particular to the case where $R$ is the ring of integers of an algebraic number field. The author recently showed an analogous result for the case where $R$ is a ring of $S$-integers of an algebraic function field of one variable over a perfect field \cite{Frei2010}.

A related question is whether or not a ring $R$ is generated by its units. If we take $R$ to be a ring of integers of an algebraic number or function field, both possibilities occur. Complete classifications have been found in many special cases, including rings of integers of quadratic number fields \cite{Ashrafi2005, Belcher1974} and certain types of cubic and quartic number fields \cite{Filipin2008, Tichy2007, Ziegler2008}, and rings of $S$-integers of quadratic function fields \cite{Frei2010}. All of these results have in common that the unit group of the ring in question is of rank $1$. The author is not aware of any general results for rings of integers whose unit groups have higher rank.

Among other problems, Jarden and Narkiewicz asked the following question, which was later called the extension problem.

\begin{problem}\cite[Problem B]{Jarden2007}\label{problem}
Is it true that each number field has a finite extension $L$ such that the ring of integers of $L$ is generated by its units.
\end{problem}

This is of course true for finite abelian extensions of $\Q$, since those are contained in cyclotomic number fields by the Kronecker-Weber theorem, and the ring of integers of a cyclotomic number field is generated by a root of unity. The scope of this paper is an affirmative answer to the function field version of Problem \ref{problem}. Let us fix some basic notation before we state the theorem.

Regarding function fields, we use the notation from \cite{Rosen2002} and \cite{Stichtenoth1993}. In particular, an algebraic function field over a field $K$ is a finitely generated extension $F | K$ of transcendence degree $1$. The algebraic closure of $K$ in $F$ is called the (full) constant field of $F | K$. An element $t \in F$ is called a separating element for $F | K$, if the extension $F | K(t)$ is finite and separable. Following \cite{Stichtenoth1993}, we regard the places $P$ of $F | K$ as the maximal ideals of discrete valuation rings $\O_P$ of $F$ containing $K$. In particular, the places correspond to (surjective) discrete valuations $v_P : F \to \Z \cup \{\infty\}$ of $F$ over $K$. Let $n$ be a positive integer. We say that a place $P$ of $F | K$ is a zero of an element $f \in F$ of order $n$, if $v_P(f) = n > 0$, and $P$ is a pole of $f$ of order $n$, if $v_P(f) = -n < 0$. If $S$ is a finite set of places of $F | K$ then the ring $\O_S$ of $S$-integers of $F$ is the set of all elements of $F$ that have no poles outside of $S$. Moreover, we write $K^\times := K \smallsetminus\{0\}$.

\begin{theorem}\label{extension_problem_function_field}
Let $K$ be a perfect field, $F | K$ an algebraic function field over $K$, and $S \neq \emptyset$ a finite set of places of $F | K$. Let $\O_S$ be the ring of $S$-integers of $F$. Then there exists a finite extension $F'|F$ such that the integral closure of $\O_S$ in $F'$ is generated by its units (as a ring). 
\end{theorem}

The basic idea to prove Theorem \ref{extension_problem_function_field} is the following: First, choose a finite set $\{t, t_1, \ldots, t_n\}$ of generators  of $\O_S$ over $K$. Then, for each $1 \leq i \leq n$, iteratively construct a finite extension $F_i | F$ such that
\begin{enumerate}[(I.)]
 \item $t$, $t_1$, $\ldots$, $t_i$ are sums of units in the integral closure of $\O_S$ in $F_i$, and
 \item the integral closure of $\O_S$ in $F_i$ is generated by units as a ring extension of $\O_S$.
\end{enumerate}
Then the integral closure of $\O_S$ in $F_n$ is generated by units and sums of units as an extension of $K$, thus it is generated by its units. Section \ref{auxiliary} provides the tools to construct the extension fields $F_i$. In Section \ref{proof}, everything is put together.

\section{Auxiliary results}\label{auxiliary}

The following lemma illustrates the idea explained at the end of the introduction.

\begin{lemma}\label{make_t_unit_sum}
Let $K$ be a perfect field not of characteristic $2$ and $a \in K^\times$. Consider the extension of rational function fields $K(x) | K(t)$, where $t = x + a^2 / x$. Then the integral closure of $K[t]$ in $K(x)$ is $K[x, x^{-1}]$, which is generated (as a ring) by its units.

The only places of $K(t)$ that are ramified in $K(x)$ are the zeros of $t - 2 a$ and $t + 2 a$, both with ramification index $2$. 
\end{lemma}

\begin{proof}
The minimal polynomial of $x$ over $K(t)$ is $X^2 - t X + a^2$, whence $K(x) = K(t, y)$, with $y^2 = t^2 - 4 a^2$. (Here we used the assumption that $K$ is not of characteristic $2$.) One can verify the assertions about ramification directly or use Proposition III.7.3 from \cite{Stichtenoth1993}.

Obviously, $x$ and $x^{-1}$ are integral over $K[t]$, and $K[t] \subseteq K[x, x^{-1}]$. Since $K[x, x^{-1}]$, as a ring of fractions of the principal ideal domain $K[x]$, is integrally closed, it is the integral closure of $K[t]$ in $K(x)$. Obviously, $x$, $x^{-1}$ and all elements of $K^\times$ are units in $K[x, x^{-1}]$, and the lemma is proved.
\end{proof}

The main step in the construction of the extension fields $F_i$ is carried out in the following proposition, which is the most important component of our proof of Theorem \ref{extension_problem_function_field}.

\begin{proposition}\label{main_work}
Let $K$ be a perfect field not of characteristic $2$, $F | K$ an algebraic function field with full constant field $K$, $t$ a separating element of $F | K$, and $\O$ the integral closure of $K[t]$ in $F$. Assume that there is some $a \in K^\times$ such that the zeros of $t + 2a$ and $t - 2a$ in $K(t)$ are unramified in $F | K(t)$. 

Let $F' := F(x)$, where $x$ is a root of the polynomial $f := X^2 - t X + a^2$, and let $\O'$ be the integral closure of $K[t]$ in $F'$. Then $K$ is the full constant field of $F' | K$, $x$ is a unit in $\O'$, $t = x + a^2/x$, and $\O' = \O[x]$. 
\end{proposition}

\begin{proof}
The roots of $f \in \O[X]$ in $F'$ are $x$ and $a^2 / x$, whence $x$ is a unit in $\O'$. Obviously, $t = x + a^2/ x$. If $f$ is reducible over $F$ then $x, a^2/x \in \O$, and the proposition holds trivially. Assume now that $f$ is irreducible over $F$. 

The field $F'$ is the compositum of $F$ and $K(x)$. Since the characteristic of $K$ is not $2$, the extension $F' | F$, and thus as well $F' | K(t)$ is separable. By Lemma \ref{make_t_unit_sum}, the only places of $K(t)$ that are ramified in $K(x)$ are the zeros of $t-2a$ and $t+2a$, both with ramification index $2$.

Let $P$ be a zero of $t + 2a$ or $t - 2a$ in $F' | K$. By Abhyankar's lemma (see, for example, Proposition III.8.9 from \cite{Stichtenoth1993}), the ramification index of $P$ over $K(t)$ is $2$. Here, we used the assumption that the zeros of $t - 2a$ and $t + 2a$ in $K(t)$ are unramified in the extension $F | K(t)$. Therefore, the ramification index of $P$ over $F$ is $2$. 

Again by Abhyankar's lemma, every place $Q$ of $F' | K$ that is not a zero of $t + 2a$ or $t - 2a$ is unramified over $F$.

Since there are ramified places in the extension $F' | F$, it is not a constant field extension, so $K$ is the full constant field of $F' | K$.

We are left with the task of proving that $\O' = \O[x]$. Denote the different of $\O' | \O$ by $\D$, and let $\delta(x)$ be the different of $x$, that is $\delta(x) = f'(x) = 2x - t$. It is well known that $\O' = \O[x]$ if and only if $\D$ is the principal ideal of $\O'$ generated by $\delta(x)$ (see, for example, Theorem V.11.29 from \cite{Zariski1975}). 

Already knowing all ramification indices in the extension $F' | F$, we see that the different $\D$ of $\O' | \O$ is the product of all prime ideals of $\O'$ dividing $(t + 2a)$ or $(t - 2a)$ (use, for example, Theorem III.2.6 from \cite{Neukirch1999} and the assumption $K$ is not of characteristic $2$).

Since
$$\delta(x)^2 = (2 x-t)^2 = t^2 - 4a^2 = (t + 2a)(t - 2a)\text,$$
the ideal of $\O'$ generated by $\delta(x)$ satisfies
$$(\delta(x))^2 = \prod_{\P | (t \pm 2a)}\P^2 = \left(\prod_{\P | (t \pm 2a)} \P\right)^2 = \D^2\text.$$
Here, $\P$ ranges over all prime ideals of $\O'$ dividing $(t + 2 a)$ or $(t - 2 a)$. As we have already seen, the ramification index of each such $\P$ over the prime ideal $(t + 2 a)$ [or $(t - 2a)$] of $K[t]$ is $2$. By unique ideal factorization, the ideal of $\O'$ generated by $\delta(x)$ is $\D$.
\end{proof}

For function fields of characteristic $2$, we use a slightly modified form of Proposition \ref{main_work}.

\begin{proposition}\label{main_work_char_2}
Let $K$ be a perfect field of characteristic $2$, $F | K$ an algebraic function field with full constant field $K$, $t$ a separating element of $F | K$, and $\O$ the integral closure of $K[t]$ in $F$. Assume that there is some $a \in K$ such that the zero of $t + a$ in $K(t)$ is unramified in $F | K(t)$. 

Let $F' := F(x)$, where $x$ is a root of the polynomial $f := X^2 + (t + a) X + 1$, and let $\O'$ be the integral closure of $K[t]$ in $F'$. Then $K$ is the full constant field of $F' | K$, $x$ is a unit in $\O'$, $t = x + 1/x + a$, and $\O' = \O[x]$. 
\end{proposition}

\begin{proof}
Again, $x$ is a unit in $\O'$, since $x$ and $1/x$ are the roots of the monic polynomial $f \in \O[X]$. Clearly, $t = x + 1/x + a$. The proposition holds again trivially if $f$ is reducible over $F$. Assume from now on that $f$ is irreducible over $F$. 

Putting $y := x/(t + a)$, we get $F' = F(x) = F(y)$ and  $y^2 + y = 1/(t+a)^2$. We use Proposition III.7.8 from \cite{Stichtenoth1993} to prove that the only places of $F|K$ that are ramified in $F'$ are the zeros of $t + a$. Indeed, for each such zero $P$, we have 
$$v_P\left(1/(t+a)^2 - (1/(t+a)^2 - 1/(t+a))\right) = v_P(1/(t+a)) = -1\text,$$
since $P$ is unramified over $K(t)$. For each place $Q$ of $F|K$ that is not a zero of $t + a$, we have
$$v_Q(1/(t + a)^2) \geq 0\text.$$
Therefore, Proposition III.7.8 from \cite{Stichtenoth1993} implies that the places of $F|K$ that are ramified in $F'$ are exactly the zeros of $t + a$, and that the respective ramification indices and different exponents are $2$. We conclude that $K$ is the full constant field of $F | K$ and that the different $\D$ of $\O' | \O$ is of the form
$$\D = \prod_{\P | (t + a)} \P^2\text.$$
Here, $\P$ ranges over all prime ideals of $\O'$ dividing $(t + a)$. On the other hand, the different of $x$ is $\delta(x) = f'(x) = t + a$, and the ideal of $\O'$ generated by $t + a$ is given by
$$(t + a) = \prod_{\P | (t + a)} \P^2 = \D\text.$$
Note that the ramification index of every prime ideal $\P$ of $\O'$ over the prime ideal $(t + a)$ of $K[t]$ is $2$, since $(t + a)$ is unramified in $F$ and the ramification index of $\P$ over $F$ is $2$.

Therefore, $\D = (\delta(x))$, which suffices to prove that $\O' = \O[x]$.
\end{proof}

The following lemma shows a way to enlarge $\O$, while still maintaining the property that $\O' = \O[x]$ from the previous propositions. The results are probably not new, but the author is not aware of an adequate reference. Recall that, for any place $P$ of an algebraic function field, $\O_P$ denotes the discrete valuation ring with maximal ideal $P$.

\begin{lemma}\label{integrally_closed}
Let $F | K$ be an algebraic function field with perfect constant field $K$, $F'| F$ a finite separable extension, and $x \in F'$ with $F' = F(x)$. Let $S \subseteq T$ be sets of places of $F | K$, and assume that $x$ is integral over $\O_S$. Then we have:
\begin{enumerate}[(a)]
 \item If $\O_P[x]$ is integrally closed for all $P \notin S$ then $\O_S[x]$ is integrally closed as well.
 \item If $\O_S[x]$ is integrally closed then $\O_T[x]$ is integrally closed as well.
 \item If $x$ is algebraic over $K$ then $\O_T[x]$ is integrally closed.
\end{enumerate}
\end{lemma}

\begin{proof}
Denote the integral closure of $\O_S$ in $F'$ by $\O'$. Clearly, $\O_S[x] \subseteq \O'$. To prove \emph{(a)}, we need to show that $\O_S[x] = \O'$. Let $S'$ be the set of places of $F'|K$ lying over places in $S$. We have
$$\O' = \bigcap_{P' \notin S'} \O_{P'} = \bigcap_{P \notin S} \bigcap_{P' | P} \O_{P'} = \bigcap_{P \notin S} (\O_P[x])\text.$$
Here, $P'$ denotes places of $F'|K$ and $P$ denotes places of $F | K$. The third equality follows from the assumption that $\O_P[x]$ is integrally closed and the fact that $x$ is integral over $\O_P$, for all $P \notin S$. Therefore, it is sufficient to show that
$$\bigcap_{P \notin S}(\O_P[x]) = \left(\bigcap_{P \notin S}\O_P\right)[x]\text.$$
Clearly, the right-hand side of the above equality is included in the left-hand side. Now let $f$ be an arbitrary element of $\bigcap_{P \notin S}(\O_P[x])$. Denote the degree $[F' : F]$ by $n$. Then, for each $P \notin S$, there is some polynomial $g_P \in \O_P[X]$ of degree smaller than $n$, with $f = g_P(x)$. Since $\{1, x, \ldots, x^{n-1}\}$ is a basis of $F' | F$, all $g_P$ are equal and thus elements of $\left(\bigcap_{P \notin S}\O_P\right)[X]$. This shows the other inclusion.

To prove \emph{(b)}, notice that, for all $P \notin S$, $\O_P$ is the localization of $\O_S$ at the unique prime ideal $\P$ of $\O_S$ corresponding to the place $P$. Therefore, $\O_P[x]$ can be seen as ring of fractions of $\O_S[x]$ with denominators in the multiplicative set $\O_S \smallsetminus \P$. Assume that $\O_S[x]$ is integrally closed. By the above argument, $\O_P[x]$ is integrally closed for all $P \notin S$, in particular for all $P \notin T$, so \emph{(b)} follows from \emph{(a)}.

The special case of $\emph{(b)}$ with $S = \emptyset$ is exactly $\emph{(c)}$.
\end{proof}

As an immediate consequence of Lemma \ref{integrally_closed} \emph{(c)} and the primitive element theorem, we get that finite constant field extensions have property (II.) from the overview presented at the end of Section \ref{introduction} (see also the third paragraph of Remark 6.1.7 from \cite{Fried2008} for a more general formulation):

\begin{corollary}\label{integral_closure_constant_field_extension}
Let $F | K$ be an algebraic function field with perfect constant field $K$, $S$ a set of places of $F | K$, and $K' | K$ a finite extension. Then the integral closure of $\O_S$ in $K'F$ is $K'\O_S$.
\end{corollary}

To use Propositions \ref{main_work} and \ref{main_work_char_2}, we need to ensure that we can always find an $a$ as required. This is accomplished by the following lemma.

\begin{lemma}\label{tplusminus2a}
Let $F | K$ be an algebraic function field with perfect constant field $K$, and $t \in F \smallsetminus K$. Then there is a finite extension $K_0 | K$ and an element $a \in K_0^\times$, such that the zeros of $t-a$ and $t+a$ in $K_0(t)$ are unramified in the extension $K_0 F | K_0(t)$.

If $F$ is separable over $K(t)$ then $K_0 F$ is separable over $K_0(t)$.
\end{lemma}

\begin{proof}
The first part of the lemma clearly holds if $K$ is infinite, since there are only finitely many ramified places in $F | K(t)$, so we can put $K_0 := K$.

In the general case, consider the algebraic closure $\ol{K}$ of $K$ in some algebraically closed field $\Phi \supseteq F$ and the constant field extension $\ol{K} F | \ol{K}$ of $F | K$. Since $\ol{K}$ is infinite, we find some $a \in \ol{K}$, such that the zeros of $t - a$ and $t + a$ in $\ol{K}(t)$ are unramified in $\ol{K} F$. Put $K_0 := K(a)$. Then the zeros of $t - a$ and $t + a$ in $K_0(t)$ are unramified in $K_0 F$, as desired. Indeed, let $\ol{P}'$ be a place of $\ol{K} F | \ol{K}$ lying over the zero $P$ of, say, $t + a$ in $K_0(t)$. Put $P' := \ol{P}' \cap K_0 F$ and $\ol{P} := \ol{P}'\cap \ol{K}(t)$. We know that $\ol{P}' | \ol{P}$ is unramified. From $\ol{P}' | \ol{P} | P$ and the fact that constant field extensions are unramified, it follows that $\ol{P}' | P$ is unramified. Now $\ol{P}' | P' | P$ implies that $P' | P$ is unramified.

The assertion regarding separability holds because if $F$ is separable over $K(t)$ then $K_0 F$ is generated over $K_0(t)$ by separable elements.
\end{proof}

\section{Proof of Theorem 2}\label{proof}
For convenience, let us state the theorem again.
\setcounter{theorem}{1}
\begin{theorem}
Let $K$ be a perfect field, $F | K$ an algebraic function field over $K$, and $S \neq \emptyset$ a finite set of places of $F | K$. Let $\O_S$ be the ring of $S$-integers of $F$. Then there exists a finite extension $F'|F$ such that the integral closure of $\O_S$ in $F'$ is generated by its units (as a ring). 
\end{theorem}

It is enough to prove Theorem \ref{extension_problem_function_field} under the assumption that $K$ is the full constant field of $F|K$, since then the general case follows as well.

Denote the characteristic of $K$ by $p \geq 0$, and assume first that $p \neq 2$. We find a separating element $t$ of $F | K$ such that $\O_S$ is the integral closure of $K[t]$ in $F$. To this end, choose places $Q \in S$ and $R$, $R' \notin S$ of $F|K$. By the strong approximation theorem, we can find an element $t \in F$ that satisfies the conditions
\begin{align*}v_R(t) &= 1\text,\\v_{R'}(t) &= \sum_{P \in S\smallsetminus\{Q\}}\deg P\text,\\v_P(t) &= -1\text{, for all } P\in S\smallsetminus\{Q\}\text{, and}\\v_P(t) &\geq 0\text{, for all places } P\notin S\cup \{R, R'\}\text{.}
\end{align*}
Since the principal divisor of $t$ has degree $0$, it follows that $v_Q(t) < 0$. Therefore, the poles of $t$ are exactly the elements of $S$. Moreover, $t$ is not a $p$-th power, since $p$ does not divide $v_R(t) = 1$. It follows that $F$ is separable over $K(t)$ (see, for example, Proposition III.9.2 (d) from \cite{Stichtenoth1993}) and the integral closure of $K[t]$ in $F$ is exactly $\O_S$.

Choose some non-constant elements $t_1, \ldots, t_n$ of $\O_S$, such that $\O_S = K[t, t_1, \ldots, t_n]$ (for example, let $\{t_1, \ldots, t_n\}$ be an integral basis of $\O_S$ over $K[t]$ and omit a possible constant).

Lemma \ref{tplusminus2a} permits us to find a finite extension $K_0 | K$ and some $a \in K_0^\times$, such that the zeros of $t - 2a$ and $t + 2a$ in $K_0(t)$ are unramified in $K_0 F$.  By Corollary \ref{integral_closure_constant_field_extension}, the integral closure of $\O_S$ in $K_0 F$ is $K_0 \O_S = K_0[t, t_1, \ldots, t_n]$. 

Proposition \ref{main_work} yields a finite extension $F_0 | K_0$ of $K_0 F | K_0$, such that $t$ is a sum of units in the integral closure $\O_0$ of $\O_S$ in $F_0$, and $\O_0 = K_0 \O_S [x_0] = K_0[t, t_1, \ldots, t_n, x_0]$, for some unit $x_0$ of $\O_0$. Moreover, $K_0$ is the full constant field of $F_0 | K_0$.

We inductively construct finite extensions $F_1 | K_1$, $\ldots$, $F_n | K_n$ of $F_0 | K_0$ with the following properties. If $\O_i$ denotes the integral closure of $\O_S$ in $F_i$ then we have, for $i \in \{0, \ldots, n\}$: 
\begin{itemize}
 \item $\O_{i} = K_i[t, s_1, \ldots, s_i, t_{i+1}, \ldots t_n, x_0, x_1, \ldots, x_i]$, where $x_0$, $\ldots$, $x_i$ are units of $\O_i$, and for all $1 \leq j \leq i$ there is some $m$ with $s_j^{p^{m}} = t_j$.
 \item $t$, $s_1$, $\ldots$, $s_i$ are sums of units of $\O_i$.
 \item $K_i$ is the full constant field of $F_i | K_i$.
\end{itemize}

For $i = 0$, the function field $F_0 | K_0$ has all desired properties. Let $i \in \{1, \ldots, n\}$ and assume that we have constructed $F_{i-1} | K_{i-1}$. The figure on page \pageref{fielddiagram} shows the relations between the rings and fields constructed in the following paragraphs.

Take the maximal non-negative integer $m$ such that $t_i$ is a $p^m$-th power in $F_{i-1}$ (the maximum exists since $t_i$ is not constant), and let $s_i$ be the $p^m$-th root of $t_i$. (If $p=0$, simply put $s_i := t_i$.) Then $s_i \in \O_{i-1}$, since $s_i$ has the same poles as $t_i$. Therefore, $\O_{i-1} = K_{i-1}[t, s_1, \ldots, s_i, t_{i+1}, \ldots, t_n, x_0, \ldots, x_{i-1}]$.

Since $s_i$ is not a $p$-th power in $F_{i-1}$, it is a separating element of $F_{i-1} | K_{i-1}$ (again, we used Proposition III.9.2 (d) from \cite{Stichtenoth1993}). By Lemma \ref{tplusminus2a}, there is some finite extension $K_i | K_{i-1}$ and some $a \in K_i^\times$ such that the zeros of $s_i - 2a$ and $s_i + 2a$ in $K_i(s_i)$ are unramified in $K_i F_{i-1}$, and $K_i F_{i-1}$ is separable over $K_i(s_i)$.

Denote the integral closure of $K_i[s_i]$ in $K_i F_{i-1}$ by $\O$. By Proposition \ref{main_work}, there is a finite extension $F_i | K_i$ of $K_i F_{i-1} | K_i$, such that the integral closure of $\O$ in $F_i$ is $\O[x_i]$, for some unit $x_i$, and $s_i$ is a sum of units in $\O[x_i]$. Moreover, $K_i$ is the full constant field of $F_i | K_i$. 

By our convention, $\O_i$ is the integral closure of $\O_S$ in $F_i$, and thus as well the integral closure of $\O_{i-1}$ in $F_i$. By Corollary \ref{integral_closure_constant_field_extension}, the integral closure of $\O_{i-1}$ in $K_i F_{i-1}$ is $K_i \O_{i-1}$. Since $s_i \in \O_{i-1}$, we have $\O \subseteq K_i \O_{i-1}$. Let $U$ be the set of poles of $s_i$ in $K_i F_{i-1}$, and $V \supseteq U$ the set of poles of $t$ in $K_i F_{i-1}$. Then $\O = \O_U$ and $K_i \O_{i-1} = \O_V$.  Since $\O_U[x_i] = \O[x_i]$ is integrally closed, Lemma \ref{integrally_closed} \emph{(b)} implies that $\O_V[x_i] = K_i \O_{i-1}[x_i]$ is integrally closed as well. Therefore, $K_i \O_{i-1} [x_i]$ is $\O_i$, the integral closure of $\O_{i-1}$ in $F_i$. We conclude that 
$$\O_i = K_i[t, s_1, \ldots, s_i, t_{i+1}, \ldots, t_n, x_0, \ldots, x_i]\text,$$
as desired. The elements $x_0$, $\ldots$, $x_{i-1}$ are units in $\O_i$, because they are units in $\O_{i-1} \subseteq \O_i$. Moreover, $x_i$ is a unit in $\O_i$, since it is a unit in $\O[x_i] \subseteq \O_i$. Therefore, $t$, $s_1$, $\ldots$, $s_i$ are sums of units of $\O_i$, and the induction is complete.

\begin{figure}\label{fielddiagram}
$$\xymatrix@C-=5pt@R-=5pt{
**[l] F_i\ar@{-}[dd]_{\text{ Proposition 4}} & & & & **[r]F_i \ar@{-}[dd]\\
& \O[x_i] \ar@{-}[dd] \ar@{-}[lu] &\subseteq& \O_i \ar@{-}[dd] \ar@{-}[ru]&  &\\
**[l]K_i F_{i-1} \ar@{-}[dd]_{\text{ Lemma 8}}& & & & **[r]K_i F_{i-1} \ar@{-}[dd]\\
& \O \ar@{-}[dd] \ar@{-}[lu] &\subseteq& K_i \O_{i-1} \ar@{-}[dd] \ar@{-}[ru]&  &\\
**[l]K_i(s_i) & & & & **[r]F_{i-1}\\
& K_i[s_i] \ar@{-}[lu]& & \O_{i-1} \ar@{-}[ru]&  &\\
}$$
\caption{The rings and fields occurring in the induction step.}
\end{figure}
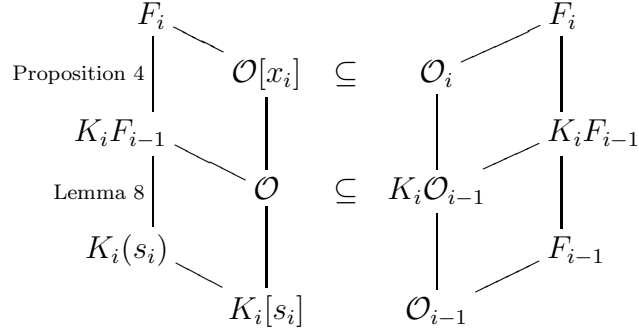

Now put $F' | K' := F_n | K_n$, and Theorem \ref{extension_problem_function_field} is proved whenever the characteristic of $K$ is not $2$. In characteristic $2$, the proof is exactly the same as above, except that we always write $a$ instead of $2a$ and use Proposition \ref{main_work_char_2} instead of Proposition \ref{main_work}.

\subsection*{Acknowledgements}
The author is supported by the Austrian Science Foundation (FWF) project S9611-N23.

\bibliographystyle{plain}
\bibliography{extension_problem}

\noindent Christopher Frei\\
Technische Universit\"at Graz\\
Institut f\"ur Analysis und Computational Number Theory\\
Steyrergasse 30, 8010 Graz, Austria\\
E-mail: frei@math.tugraz.at\\
\url{http://www.math.tugraz.at/~frei}

\end{document}